\documentclass[11pt,a4paper,leqno,twoside,notitlepage, headings=small]{article}

\usepackage{amsmath}
\usepackage{amsfonts}
\usepackage{amssymb}
\usepackage{amsxtra}
\usepackage{amsthm}
\usepackage{enumerate}
\usepackage{ifthen}

\newboolean{details}

\theoremstyle{plain}
\newtheorem{thm}[equation]{Theorem}
\newtheorem*{thm*}{Theorem}
\newtheorem{prop}[equation]{Proposition} 
\newtheorem{lem}[equation]{Lemma}

\theoremstyle{definition}
\newtheorem{defn}[equation]{Definition}
\newtheorem*{defn*}{Definition}
\newtheorem*{ex*}{Example}

\newcommand{\CC}{{\mathbb C}}

\newcommand{\RR}{{\mathbb R}}
\newcommand{\ZZ}{{\mathbb Z}}

\newcommand{\mcD}{\mathcal{D}}

\newcommand{\mcX}{\mathcal{X}}

\newcommand{\frakg}{\mathfrak{g}}
\newcommand{\frakgl}{\mathfrak{gl}}
\newcommand{\frakh}{\mathfrak{h}}

\newcommand{\fraks}{\mathfrak{s}}
\newcommand{\frakt}{\mathfrak{t}}

\newcommand{\fraku}{\mathfrak{u}}

\newcommand{\GL}{GL}
\newcommand{\SL}{SL}

\renewcommand{\to}{\longrightarrow}

\providecommand{\Ad}{\mathop{\rm Ad}\nolimits}

\providecommand{\ad}{\mathop{\rm ad}\nolimits}

\providecommand{\Id}{\mathop{\rm Id}\nolimits}

\providecommand{\Aut}{\mathop{\rm Aut}\nolimits}

\providecommand{\lspan}{\mathop{\rm span}\nolimits}

\providecommand{\pr}{\mathop{\rm pr}\nolimits}

\providecommand{\tr}{\mathop{\rm tr}\nolimits}

\setboolean{details}{true}

\begin{document}

\centerline{\LARGE\bfseries{Rigidity of lattices and syndetic hulls}}
\vspace{0.25cm}
\centerline{\LARGE\bfseries{in solvable Lie groups}}
\vspace{1cm}
\centerline{\large Oliver Ungermann}
\vspace{0.75cm}
\centerline{Institut f\"ur Mathematik und Informatik}
\centerline{Universit\"at Greifswald}
\vspace{0.75cm}
\centerline{\large November 2013}
\vspace{1cm}

\begin{abstract}
First, let $G$ be a completely solvable Lie group. We recall the proof of the
fact that any closed subgroup of $G$ possesses a unique syndetic hull in $G$.
As a consequence we obtain that strong rigidity holds for the class of completely
solvable Lie groups in the sense of G.\ D.\ Mostow: Let $G_1$, $G_2$ be completely
solvable, $\Gamma$ a uniform subgroup of $G_1$ and $\alpha:\Gamma\to G_2$ a
homomorphism of Lie groups such that $\alpha(\Gamma)$ is uniform in $G_2$. Then
there is an isomorphism $\varphi:G_1\to G_2$ such that $\varphi\,|\,\Gamma=\alpha$.
Now let $G$ be an arbitrary (exponential) solvable Lie group. We will discuss certain
conditions on closed subgroups of~$G$ which are sufficient for the existence of
a syndetic hull.
\end{abstract}

\vspace{1cm}

\subsection*{\large Introduction}
Let $G$ be an exponential solvable Lie group with Lie algebra~$\frakg$. This means
that the exponential mapping gives a diffeomorphism of~$\frakg$ onto~$G$. In particular
$G$ is simply connected.  We say that a solvable Lie group is completely solvable if
it is simply connected and all its roots are real. Every completely solvable Lie group
is exponential. In Section~1 we will recall the proof that normalizers of connected Lie
subgroups of completely solvable Lie groups are connected.\\

\begin{defn*}
Let $\Gamma$ be a closed subgroup of a Lie group $G$. A connected Lie subgroup $S$ of $G$ with
$\Gamma\subset S$ and $\Gamma\setminus S$ compact is called a syndetic hull of $\Gamma$ in $G$.
\end{defn*}

This definition implies that $S$  is a closed subgroup of $G$ as the preimage of the compact
set $\Gamma\setminus S$ under the projection $G\to\Gamma\setminus G$.\\

At the end of Section~2 we will deduce the proof of the following result.

\begin{thm*}
Every closed subgroup of a completely solvable Lie group possesses a unique syndetic hull.
\end{thm*}

In Section~3 we introduce the new notions of algebraically dense and algebraically connected
subgroups of arbitary (exponential) solvable Lie groups which turn out to be sufficient conditions
for the existence of syndetic hulls.

\begin{defn*}
A subgroup $\Gamma$ of a Lie group $G$ is called uniform if $\Gamma\setminus G$ is a compact
Hausdorff space.
\end{defn*}

Uniform subgroups are closed.\\

Recall that a subgroup $\Gamma$ of $G$ is called analytically dense if the only connected Lie subgroup
of $G$ containing $\Gamma$ is $G$ itself. This property was studied by R.\ D.\ Mosak and M.\ Moskowitz
in~\cite{MM1}. It is easy to see that uniform subgroups of simply connected solvable Lie groups are
analytically dense, compare also theorem~8 of~\cite{MM1}.\\

Following the argument of D.\ Witte Morris in Section~3 of~\cite{W2} and using the preceding
theorem, we obtain that strong rigidity in the sense of Mostow's definition in~\cite{M}
holds for the class of completely solvable Lie groups.

\begin{thm*}
Let $G_1$, $G_2$ be completely solvable Lie groups of the same dimension and $\Gamma$ a uniform
subgroup of~$G_1$. Let $\alpha:\Gamma\to G_2$ be a homomorphism of Lie groups such that $\alpha(\Gamma)$
is uniform in $G_2$. Then there exists a unique isomorphism $\varphi:G_1\to G_2$ such that
$\varphi\,|\,\Gamma=\alpha$.
\end{thm*}
\begin{proof}
We consider the closed subgroup $\widehat{\Gamma}:=\{(x,\alpha(x)):x\in\Gamma\}$
of the exponential solvable Lie group ~$\widehat{G}:=G_1\times G_2$. The preceding theorem
implies that there exists a closed connected subgroup $\widehat{S}$ of $\widehat{G}$ such
that $\widehat{\Gamma}\setminus\widehat{G}$ is compact. Choose a compact subset $\widehat{C}$
of $\widehat{S}$ such that $\widehat{S}=\widehat{\Gamma}\widehat{C}$. Let $p:\widehat{G}\to G_1$
denote the projection onto the first factor. Since $p(\widehat{S})$ is a connected Lie subgroup
of $G_1$ containing the uniform subgroup~$\Gamma$, it follows $p(\widehat{S})=G_1$. It is easy to
see that the map $p:\widehat{S}\to G_1$ is proper: If $K\subset G$ is compact and $(x_n,y_n)\in p^{-1}(K)$
is an arbitrary sequence, then $(x_n,y_n)$ has a convergent subsequence with a limit
in $K\times\alpha(\Gamma)\pr_2(\widehat{C})$. In particular $\ker p=\widehat{S}\cap p^{-1}(\{e\})$
is a compact subgroup of $\widehat{G}$ and hence trivial. This means that for every $x\in G_1$ there
exists a unique $\varphi(x)\in G_2$ such that $(x,\varphi(x))\in\widehat{S}$. It is easy to see
that $\varphi:G_1\to G_2$ is a continuous homomorphism. This implies that $\varphi$ is smooth. Since
the image of $\varphi$ contains $\alpha(\Gamma)$, it follows that $\varphi$ is surjective. Finally,
using that $\exp_{G_2}$ is bijective and that $G_1$ is simply connected, one can prove that $\varphi$ is
injective. Hence $\varphi$ is an isomorphism.\\

Let $\varphi:G_1\to G_2$ be an arbitrary continuous homomorphism extending $\alpha$. Then
$\widehat{S}:=\{(x,\varphi(x)):x\in G_1\}$ is a closed connected subgroup of $\widehat{G}$ containing
$\widehat{\Gamma}$. Since $\widehat{S}$ is minimal with this property, it follows from the
preceding theorem that $\widehat{S}$ is the unique syndetic hull of $\widehat{S}$ in $\widehat{G}$.
This proves the uniqueness of $\varphi$.
\end{proof}

In particular, we see that the assumptions on $\alpha$ imply that $\alpha$ is injective.\\

Let $\Gamma$ be a uniform subgroup of a Lie group $G$. Let $\mcX(\Gamma,G)$ denote the set
of all injective continuous homomorphisms $\alpha:\Gamma\to G$ such that $\alpha(\Gamma)$ is uniform
in~$G$, equipped with the topology of compact convergence. The automorphism group of $G$ acts
continuously on $\mcX(\Gamma,G)$ from the left by composition. The orbit space
$\mcD(\Gamma,G):=\Aut(G)\setminus\mcX(\Gamma,G)$ is called the deformation space of the embedding
of $\Gamma$ in $G$.  The preceding theorem states that $\mcD(\Gamma,G)$ consists of a single point
provided that $G$ is completely solvable. In~\cite{BKl}, the deformation space has been studied
for uniform discrete subgroups of arbitrary simply connected solvable Lie groups.\\

The preceding theorem applies in particular to uniform discrete subgroups of $G$, sometimes
called lattices. The existence of a lattice of $G$ implies that $G$ is unimodular. Note that
there exist many examples of exponential solvable Lie groups which admit lattices and which are
not completely solvable.\\

\begin{ex*}
Let $A=\left(\begin{array}{cc} a & b\\ c &d \end{array}\right)\in\SL(2,\ZZ)$. Then $\det A=ad-bc=1$
is the determinant of~$A$, and $\chi_A(\lambda)=\lambda^2-\tr(A)\lambda+1$ the characteristic polynomial.
Clearly $A$ has non-real eigenvalues which are not purely imaginary if and only if $|\tr A|=1$. Let
us fix such a matrix, e.g., $A=\left(\begin{array}{cc} 1 & -1\\ 1 & 0 \end{array}\right)$.
Then there exists $B\in\frakgl(2,\RR)$ such that $\exp(B)=A$. We form the semi-direct
product $G:=\RR^2\ltimes_A\RR$ by means of the one-parameter group $A(t)=\exp(tB)$. 
By definition $G$ is an exponential solvable Lie group which is not completely solvable and
which contains the subgroup $\Gamma:=\ZZ^2\ltimes_A\ZZ$ as a uniform discrete subgroup.
\end{ex*}

Milovanov claims that there exist lattices of exponential solvable Lie group which are not strongly
rigid, see example~2.9 of Starkov's article~\cite{S}.

\subsection{Normalizers and centralizers}

In this section we will prove that normalizers and centralizers of connected Lie subgroups of
completely solvable Lie groups are closed and connected.

\begin{defn}\label{D1_1}
Let $G$ be a Lie group. If $\Gamma$ is a subset of $G$, then the subgroup
\[N_G(\Gamma):=\{x\in G: g\Gamma g^{-1}=\Gamma\}\]
is called the normalizer of $\Gamma$ in $G$, and the subgroup
\[C_G(\Gamma):=\{x\in G:xg=gx\text{ for alle }g\in\Gamma\}\]
is called the centralizer of $\Gamma$ in $G$. In particular $Z(G)=C_G(G)$ is the center of $G$.
Obviously $C_G(\Gamma)$ is closed. Further $N_G(\Gamma)$ is closed provided that $\Gamma$ is closed.
\end{defn}

\vspace{0.5cm}

We begin with three very simple observations.

\begin{lem}\label{L1_1}
Let $G$ and $H$ be topological groups and $\rho:G\to H$ a surjective homomorphism which is also
a quotient map. (This implies that $\rho$ is an open mapping.) Let $Q$ be a subgroup of $H$ and
$P:=\rho^{-1}(Q)$. Then the following holds true:
\begin{enumerate}
\item The restriction $\rho':P\to Q$ of $\rho$ is an open mapping.
\item If $Q$ and $\ker\rho$ are connected, then $P$ is connected.
\end{enumerate}
\end{lem}

\vspace{0.5cm}

\begin{lem}\label{L1_2}
Let $G$ be a Lie group and $U$ a connected Lie subgroup of~$G$. Then $N_G(U)$ is closed.
\end{lem}
\begin{proof}
Let $N_G(\fraku):=\{x\in G:\Ad(x)\fraku=\fraku\}$. Since $U$ is connected, $N_G(U)=N_G(\fraku)$.
This shows that $N_G(U)$ is closed as the $\Ad$-preimage of the Zariski-closed subgroup
$\{\varphi\in\Aut(\frakg):\varphi(\fraku)=\fraku\}$ of~$\Aut(\frakg)$.
\end{proof}

It is remarkable that for $N_G(U)$ to be closed we do not need that $U$ is closed.

\vspace{0.5cm}

\begin{lem}\label{L1_3}
Let $G$ be a Lie group and $\rho:G\to\GL(V)$ a representation in a vector space $V$ of dimension $n$.
If $X\in\frakg$ is an element of the Lie algebra of $G$ such that $d\rho(X)$ is nilpotent, then
\[d\rho(X)=\log(\rho(\exp(X)))=\sum_{k=1}^n\frac{1}{k}(-1)^{k+1}\;(\rho(\exp(X))-\Id)^k\;.\]
\end{lem}
\begin{proof}
Since $d\rho(X)$ is nilpotent, there exists a basis of $V$ with respect to which the
transformation matrix $d\rho(X)$ is strictly upper triangular. Let $\frakt$ be the Lie algebra
of all endomorphisms of $V$ whose transformation matrix is strictly upper diagonal with respect
to this basis, and $T$ the nilpotent Lie group of all operators such that the transformation
matrix of $T-\Id$ is strictly upper triangular. It is known that $\exp:\frakt\to T$ is bijective
and that the (welldefined) logarithm gives the inverse, see Proposition~I.2.7 in~\cite{HN}.
This proves the result.
\end{proof}

As a first step we consider normalizers and centralizers in connected nilpotent Lie groups.

\begin{lem}\label{L1_4}
Let $G$ be a connected nilpotent Lie group.  Then the center $Z(G)$ of $G$ is connected. If $U$
is a connected Lie subgroup of $G$, then the normalizer $N_G(U)$ and the centralizer $C_G(U)$
of $U$ in~$G$ are connected.
\end{lem}
\begin{proof}
Let $x\in C_G(U)$. Since the exponential function of $G$ is surjective, there exists an element $X\in\frakg$
with $x=\exp(X)$. From $xgx^{-1}=x$ for all $g\in U$ it follows that $\Ad(\exp(X))\,|\,\fraku=\Id$.
Since $\ad(X)$ is nilpotent, we can apply Lemma~\ref{L1_3} and obtain $\ad(X)\,|\,\fraku=0$. Since
$C_\frakg(\fraku):=\{Y\in\frakg:[Y,\fraku]=0\}$ is a subalgebra of $\frakg$, it follows that
$C_G(U)=\exp(C_\frakg(\fraku))$ is connected.\\

Let $x\in N_G(U)$ and $x=\exp(X)$ for some $X\in\frakg$. Again we apply Lemma~\ref{L1_3}: From $xUx^{-1}=U$
we obtain $\Ad(\exp(X))\cdot\fraku=\fraku$ and hence $\ad(X)\cdot\fraku=[X,\fraku]\subset\fraku$. Hence
$N_G(U)=\exp(N_\frakg(\fraku))$ is connected because $N_\frakg(\fraku):=\{Y\in\frakg:[Y,\fraku]\subset\fraku\}$
is connected as a subalgebra of $\frakg$.
\end{proof}

Now we consider exponential solvable Lie groups.

\begin{defn}\label{D1_2}
A Lie group $G$ is called exponential solvable if $G$ is solvable and the exponential map gives a
diffeomorphism of the Lie algebra $\frakg$ onto the Lie group $G$.
\end{defn}

\begin{defn}\label{D1_3}
A Lie algebra $\frakg$ is called exponential solvable if for every $X\in\frakg$ the eigenvalues of $\ad(X)$
lie in $(\CC\setminus i\RR)\cup\{0\}$.
\end{defn}

The following result goes back to Dixmier~\cite{D}.

\begin{prop}\label{P1_1}
A Lie group $G$ is exponential solvable if and only if $G$ is simply connected (and hence connected)
and its Lie algebra $\frakg$ is exponential solvable.
\end{prop}

An immediate consequence is that any connected Lie group whose Lie algebra is exponential solvable is
covered by a (unique) exponential solvable Lie group (namely by its universal covering).\\

Saito proved in th\'eor\`eme 2 of \cite{S1} that the center of an exponential solvable Lie group is connected.
Corwin and Moskowitz proved the connectedness of $Z(G)$ as a consequence of the multiplicative exponential property
satisfied by exponential solvable Lie groups, see corollary~4.7 and theorem~4.3 in~\cite{CM}. However, the
assumption that $G$ is simply connected is not necessary.

\begin{lem}\label{L1_5}
Let $G$ be a connected Lie group whose Lie algebra is exponential solvable. Then the center $Z(G)$
of $G$ is connected.
\end{lem}
\begin{proof}
Let $x\in Z(G)$ be arbitrary. Since the exponential function of $G$ is surjective, there exists
$X\in\frakg$ with $x=\exp(X)$.  Since $xgx^{-1}=g$ for all $g\in G$, it follows $\Ad(\exp(X))=\Id$.\\

By Lie's theorem we can think of $\ad(\frakg)\,|\,\frakg_\CC$ as a Lie algebra of upper triangular
complex matrices. Suppose that $\ad(X)$ is not nilpotent so that at least one entry on the diagonal
of $\ad(X)$ is non-zero. Since $\frakg$ is exponential, there exists $\lambda\in\CC\setminus i\RR$,
an $\ad(X)$-invariant subspace $W$ of $\frakg_\CC$ and a vector $Z\in\fraku_\CC\setminus W$ such that
\[\ad(X)\cdot Z\equiv \lambda Z\quad (\text{mod } W)\;.\]
This implies $\Ad(\exp(X))\cdot Z\equiv e^\lambda Z\;(\text{mod }W)$, a contradiction. This proves that
$\ad(X)$ is nilpotent. By Lemma~\ref{L1_3} it follows $\ad(X)=0$. Consequently $Z(G)=\exp(Z(\frakg))$
is connected.
\end{proof}

V.\ V.\ Gorbatsevich used in \cite{G2} that normalizers of connected Lie subgroups of linear Lie
groups of triangular matrices are connected, see Lemma~1 of~\cite{G2}, and refers to~\cite{G1} for
a proof which does not seem to exist.\\

\begin{prop}\label{P1_3}
Let $G$ be a connected Lie group whose Lie algebra is exponential solvable. If $U$ is a connected Lie
subgroup of $G$ such that the normalizer $N_G(U)=\exp(N_{\frakg}(\fraku))$ is connected, then the
centralizer $C_G(U)$ of $U$ in $G$ is connected, too. The connectedness of $N_G(U)$ is guaranteed
if $G$ is completely solvable.
\end{prop}
\begin{proof}
Let $x\in C_G(U)$. Since $C_G(U)\subset N_G(U)=\exp(\,N_{\frakg}(\fraku)\,)$, there is $X\in N_\frakg(\fraku)$
with $x=\exp(X)$. It holds $[X,\fraku]\subset\fraku$ and $\Ad(\exp(X))\,|\,\fraku=\Id$.  We repeat the
argument of the proof of Lemma~\ref{L1_5}: Think of $\ad(N_\frakg(\fraku))\,|\,\fraku_\CC$ as a Lie
algebra of upper triangular complex matrices and suppose that $\ad(X)\,|\,\fraku$ is not nilpotent.
Then there is $\lambda\in\CC\setminus i\RR$, an $\ad(X)$-invariant subspace $W$ of $\fraku_\CC$ and
$Z\in\fraku_\CC$ with $Z\not\in W$ and $\ad(X)\cdot Z\equiv \lambda Z\;(\text{mod }W)$. This implies
$\Ad(\exp(X))\cdot Z\equiv e^\lambda Z\;(\text{mod }W)$, a contradiction. Thus $\ad(X)$ is nilpotent.
By Lemma~\ref{L1_3} it follows $\ad(X)\,|\,\fraku=0$ so that $X\in C_\frakg(\fraku)$. This proves
that $C_G(U)=\exp(C_\frakg(\fraku))$ is connected.\\

Under the additional assumption that $G$ is completely solvable, it was proved in (the corollary) of
proposition~5 of~\cite{S2} that $N_G(U)=\exp(N_{\frakg}(\fraku))$ is connected.
\end{proof}

\subsection{Syndetic hulls in completely solvable Lie groups}
 
In this section we will prove that any closed subgroup $\Gamma$ of a completely solvable Lie group
has a unique syndetic hull.\\

First we consider the case of nilpotent groups. We will use the following result proved in
chapter~2 of~\cite{R}.

\begin{prop}\label{P2_1}
Let $G$ be a connected unipotent Lie subgroup of $\GL(n,\RR)$. Then $G$ is simply connected. Moreover,
$G$ is Zariski-closed and hence an algebraic group. Further let  $\Gamma$ be a closed subgroup
of $G$. Then $\Gamma\setminus G$ is compact if and only if $G$ is contained in the Zariski-closure
of $\Gamma$ in $\GL(n,\RR)$.
\end{prop}

Using this proposition we can prove

\begin{thm}\label{T2_1}
Let $G$ be a connected nilpotent Lie group and $\Gamma$ a closed subgroup of $G$.
Then there exists a syndetic hull of $\Gamma$ in~$G$. The syndetic hull of $\Gamma$ is
unique if $G$ is simply connected.
\end{thm}
\begin{proof}
First we consider the case that $G$ is simply connected. Then we can choose a faithful
representation $\rho:G\to\GL(n,\RR)$ such that $\dot{G}:=\rho(G)$ is a unipotent subgroup of $\GL(n,\RR)$.
Define $\dot{\Gamma}=\rho(\Gamma)$. Let $\dot{S}=\underline{\dot{\Gamma}}$ denote the Zariski-closure
of $\dot{\Gamma}$ in $\GL(n,\RR)$. Applying Proposition~\ref{P2_1} to $\dot{\Gamma}\subset\dot{S}$,
we find that $\dot{\Gamma}\setminus\dot{S}$ is compact. Since $\rho:G\to\dot{G}$ is a topological
equivalence, it follows that $\Gamma\setminus S\cong\dot{\Gamma}\setminus\dot{S}$ is compact, too.
This means that $\dot{S}$ is a syndetic hull for $\dot{\Gamma}$ in $\dot{G}$ and that $S$
is a syndetic hull for $\Gamma$ in $G$. This proves the existence. Let $S$ be any
closed subgroup of $G$ with $\Gamma\subset S$ and $\Gamma\setminus S$ compact. Put
$\dot{\Gamma}:=\rho(\Gamma)$ and $\dot{S}:=\rho(S)$. Then $\dot{\Gamma}\setminus\dot{S}$ is compact.
By Proposition~\ref{P2_1} we know that $\dot{S}$ equals the Zariski-closure
of $\Gamma$ and that $S=\rho^{-1}(\underline{S})$. This proves the uniqueness.\\

Let $G$ be an arbitrary connected nilpotent Lie group and $p:\tilde{G}\to G$ its
universal covering. Let $\Gamma\subset G$ be closed. Define $\tilde{\Gamma}:=p^{-1}(\Gamma)$.
Then we know that there exists a closed subgroup $\tilde{S}$ of $\tilde{G}$ with
$\tilde{\Gamma}\subset\tilde{S}$ and $\tilde{\Gamma}\setminus\tilde{S}$ compact. Put $S:=p(\tilde{S})$.
Since $\Gamma\setminus S\cong\tilde{\Gamma}\setminus\tilde{S}$, we see that $S$ is a syndetic
hull for $\Gamma$ in $G$.
\end{proof}

Now we study the case of completely solvable Lie groups. To prove Theorem~\ref{T2_2}
we need the following two lemmata.

\begin{lem}\label{L2_1}
Let $G$ be a connected solvable Lie group whose Lie algebra is exponential solvable.
Let $\rho:G\to\GL(V)$ be a continuous finite-dimensional representation such that for
every $X\in\frakg$ the eigenvalues of $d\rho(X)$ lie in $\{0\}\cup(\CC\setminus i\RR)$. If
$X\in\frakg$ and $v\in V$ such that $\rho(\exp(X))\cdot v=v$, then $d\rho(X)\cdot v=0$.
\end{lem}

This lemma states that the stabilizer $G_v:=\{x\in G:\rho(x)\cdot v=v\}$ of $v$ in $G$
is connected as $G_v=\exp(\frakg_v)$ with $\frakg_v:=\{X\in\frakg:d\rho(X)\cdot v=0\}$. The proof
of Lemma~\ref{L2_1} is similar to that of Proposition~\ref{P1_3}. Compare also the
proof of assertion~(43) on p.\ 49 of~\cite{LL}.

\begin{lem}\label{L2_2}
Let $G$ be an exponential solvable Lie group and $\Gamma$ be a discrete abelian subgroup
of $G$ such that $\Gamma\cap[G,G]=\{e\}$. Then there is a syndetic hull of $\Gamma$ in~$G$.
\end{lem}
\begin{proof}
Define $D:=\{X\in\frakg:\exp(X)\in\Gamma\}$ and $\fraks:=\lspan D$. We claim that $\fraks$
is a commutative Lie subalgebra of $\frakg$: Let $X,Y\in D$. Then it holds
\[\exp(\,\Ad(\exp(X))\cdot Y\,)=\exp(X)\exp(Y)\exp(-X)=\exp(Y)\;,\]
for $\Gamma$ is abelian. Since the exponential map is bijective, it follows
$\Ad(\exp(X))\cdot Y=Y$. Finally Lemma~\ref{L2_1} implies that $\ad(X)\cdot Y=[X,Y]=0$.
This proves $[\fraks,\fraks]=0$.\\

Let $S$ be the connected Lie subgroup of $G$ with Lie algebra $\fraks$. Then $S$ abelian.
If $g\in\Gamma$, then there exists some $X\in\frakg$ such that $\exp(X)=g$. It follows $X\in\fraks$
by the definition of $\fraks$. Hence $g\in S$. This proves $\Gamma\subset S$.\\

Furthermore, $\Gamma\setminus S$ is compact. To prove this, it suffices to note that $D$
is a lattice in~$\fraks$ and $D\setminus \fraks$ is compact, and that the exponential map
$\exp:\fraks\to S$ is a group homomorphism which factors to a continuous surjection
$D\setminus\fraks\to\Gamma\setminus S$.\\
\end{proof}

The idea for the proof of the next theorem is taken out of the proof of theorem~5.4 of~\cite{W2}.
Furthermore, certain aspects from~\cite{W1} were employed. When investigating discontinuous actions
on exponential solvable homogeneous spaces in~\cite{BK}, A.\ Baklouti and I.\ K\'edim indicated a
different proof for the existence of syndetic hulls.

\begin{thm}\label{T2_2}
Let $G$ be a connected Lie group whose Lie algebra is completely solvable. Let $\Gamma$ be a
closed subgroup of $G$. Then there exists a syndetic hull of $\Gamma$ in $G$. If in additon $G$
is simply connected, then the syndetic hull of $\Gamma$ is unique.
\end{thm}
\begin{proof}
First let $G$ be simply connected and hence completely solvable.  We prove the existence of
the syndetic hull in four steps.
\begin{enumerate}[1. {step:}]
\item Put $G':=[G,G]$ and $\Gamma':=\Gamma\cap G'$. Since $G'$ is a simply connected
nilpotent normal Lie subgroup of $G$, it follows from Theorem~\ref{T2_1} that there
exists a unique syndetic hull $S'$ of $\Gamma'$ in $G'$. Since $S'$ is connected,
Lemma~\ref{L1_2} and Proposition~\ref{P1_3} imply that the normalizer $N_G(S')$ is closed
and connected. For $g\in\Gamma$ it holds $g\Gamma'g^{-1}=\Gamma'$ and $gS'g^{-1}\subset G'$.
This shows that $gS'g^{-1}$ is a syndetic hull of $\Gamma'$ in $G'$. By uniqueness it
follows $gS'g^{-1}=S'$. Hence $g\in N_G(S')$. This proves $\Gamma\subset N_G(S')$. Thus
it suffices to find a syndetic hull of $\Gamma$ in $N_G(S')$.
\item Consider the group $\bar{G}:=N_G(S')/S'$ and the projection $q:N_G(S')\to\bar{G}$. Clearly
$\bar{\Gamma}:=q(\Gamma)=\Gamma S'/S'$ is a closed subgroup of $\bar{G}$ because $\Gamma S'$ is closed
in $N_G(S')$ as the preimage of the compact subset $\Gamma\setminus\Gamma S\cong \Gamma'\setminus S'$
under the projection $N_G(S')\to\Gamma\setminus N_G(S')$.\\

Suppose that $\bar{S}$ is a syndetic hull of~$\bar{\Gamma}$ in~$\bar{G}$. Put $S=q^{-1}(\bar{S})$.
By Lemma~\ref{L1_1} we know that $S$ is connected because $S'$ and $\bar{S}$ are connected.
Since $\Gamma'\setminus S'$ and $\Gamma S'\setminus S\cong\bar{\Gamma}\setminus\bar{S}$ are compact, there
exist compact sets $K'\subset S'$ and $K\subset S$ with $S'=\Gamma'K'$ and $S=\Gamma S'K$. This implies
$S=\Gamma K'K$. As $K'K$ is compact, it follows that $\Gamma\setminus S$ is compact. This means
that $S$ is a syndetic hull of $\Gamma$ in $N_G(S')$. Hence it suffices to find a syndetic
hull of $\bar{\Gamma}$ in $\bar{G}$.\\

The following observation is very important: From $p([G,G])=[\bar{G},\bar{G}]$ and $S'\subset[G,G]$
we get $p^{-1}([\bar{G},\bar{G}])=[G,G]$. Together with $p^{-1}(\bar{\Gamma})=\Gamma S'$ this gives
\[p^{-1}(\bar{\Gamma}\cap[\bar{G},\bar{G}])
=p^{-1}(\bar{\Gamma})\cap p^{-1}([\bar{G},\bar{G}])=\Gamma S'\cap[G,G]=S'\]
and hence $\bar{\Gamma}\cap[\bar{G},\bar{G}]=\{eS'\}$. In particular $[\bar{\Gamma},\bar{\Gamma}]=\{eS'\}$
is trivial so that $\bar{\Gamma}$ is abelian.
\item Let $\bar{\Gamma}_0$ be the connected component of the identity of $\bar{\Gamma}$.
Since $\bar{\Gamma}_0$ is an open normal subgroup of $\bar{\Gamma}$, we know that $\bar{\Gamma}$
is contained in the connected closed subgroup $N_{\bar{G}}(\bar{\Gamma}_0)$ of $\bar{G}$. Thus
it suffices to find a syndetic hull of $\bar{\Gamma}$ in~$N_{\bar{G}}(\bar{\Gamma}_0)$.\\

Now we consider the quotient $\tilde{G}:=N_{\bar{G}}(\bar{\Gamma}_0)/\bar{\Gamma}_0$. Clearly
$\tilde{\Gamma}:=\bar{\Gamma}/\bar{\Gamma}_0$ is a discrete subgroup of~$\tilde{G}$. If $\tilde{S}$
is a syndetic hull for $\tilde{\Gamma}$ in $\tilde{G}$, then the preimage $\bar{S}$ of $\tilde{S}$
under the projection $N_{\bar{G}}(\bar{\Gamma}_0)\to\tilde{G}$ is a syndetic hull of $\bar{\Gamma}$
in~$N_{\bar{G}}(\bar{\Gamma}_0)$: Clearly $\bar{\Gamma}\setminus\bar{S}\cong\tilde{\Gamma}\setminus\tilde{S}$
is compact and $\bar{S}$ is connected by Lemma~\ref{L1_1} because $\bar{\Gamma}_0$ and
$\tilde{S}$ are connected. Obviously the subgroup $\tilde{\Gamma}\cap[\tilde{G},\tilde{G}]=\{e\bar{\Gamma}_0\}$
is trivial because its preimage $\bar{\Gamma}\cap[\bar{G},\bar{G}]$ is trivial.
\item Thus we are in the situation of Lemma~\ref{L2_2} which guarantees the existence
of the syndetic hull of $\bar{\Gamma}$ in $\bar{G}$.
\end{enumerate}

Next we prove the uniqueness of the syndetic hull. Let $S_1$ and $S_2$ be syndetic
hulls of $\Gamma$ in $G$ with Lie algebras $\fraks_1$ and $\fraks_2$. Since $G$ is simply
connected, the exponential map of $G$ is bijective. We use this to prove that the
closed subgroup $S_1\cap S_2=\exp(\fraks_1\cap\fraks_2)$ is connected. From $\Gamma\subset S_1\cap S_2\subset S_1$
and $\Gamma\setminus S_1$ compact it follows $S_1\cap S_2=S_1$ and hence $S_1\subset S_2$. An analogous
argument yields $S_2\subset S_1$, i.e., $S_1=S_2$.\\

Finally we treat the case that the group $G$ is not simply connected. As in the proof
of Theorem~\ref{T2_1} we consider the universal covering $p:\tilde{G}\to G$ of $G$. Let
$\Gamma\subset G$ be closed. Define $\tilde{\Gamma}:=p^{-1}(\Gamma)$. We know that there
exists a closed subgroup $\tilde{S}$ of $\tilde{G}$ with $\tilde{\Gamma}\subset\tilde{S}$
and $\tilde{\Gamma}\setminus\tilde{S}$ compact. Put $S:=p(\tilde{S})$. Since
$\Gamma\setminus S\cong\tilde{\Gamma}\setminus\tilde{S}$ it follows that $S$ is a
syndetic hull of $\Gamma$ in $G$.
\end{proof}

\subsection{Syndetic hulls of subgroups of solvable Lie groups}

In this section we formulate certain conditions on closed subgroups $\Gamma$ of arbitrary (exponential)
solvable Lie groups $G$ which are sufficient for the existence of a syndetic hull.

\begin{defn}\label{D3_1}
A subgroup $\Gamma$ of a Lie group $G$ is called algebraically dense if $\Ad(G)$ is contained in
the Zariski-closure of $\Ad(\Gamma)$ taken in the algebraic group $\Aut(\frakg)$, i.e., if
$\underline{\Ad(G)}=\underline{\Ad(\Gamma)}$.
\end{defn}

We do not assume that $\Gamma$ is closed in the (Euclidean) topology of $G$.\\

At this point we refrain from defining a second topology on the Lie group $G$ using the adjoint
representation $\Ad$ and the Zariski-Topologie of $\Aut(\frakg)$. Even for nilpotent Lie groups $G$
we will not do so, neither by means of the adjoint representation nor by means of any faithful
representation, in order to avoid confusion.\\

In~\cite{MM1} Mosak and Moskowitz studied Zariski-dense subgroups $\Gamma$ of connected Lie subgroups $G$
of $\GL(n,\RR)$ and proved that Chabauty's condition holds true for solvable~$G$ by establishing that the
cohomology restriction maps $H^p(G,W)\to H^p(\Gamma,W)$ are isomorphisms for all $p$ and all rational
representations $W$ of $\underline{G}$.\\

These two notions of density are related as follows: Any Zariski-dense subgroup $\Gamma$ of a connected Lie
subgroup $G$ of $\GL(n,\RR)$ is algebraically dense because
\[\underline{\Ad(\Gamma)}=\Ad(\underline{\Gamma})=\Ad(\underline{G})=\underline{\Ad(G)}\]
where $\Ad:\underline{G}\to\Aut(\frakg)$ denotes the homomorphism of algebraic groups obtained by restricting
the coadjoint representation of $\underline{G}$ to $\frakg$. The opposite implication does not hold true.\\

An immediate consequence of $\Gamma$ being algebraically dense is that any $\Ad(\Gamma)$-invariant
subspace of the Lie algebra $\frakg$ of $G$ is $\Ad(G)$-invariant and hence an ideal of~$\frakg$.

\begin{lem}\label{L3_1}
Let $\Gamma$ be an algebraically dense subgroup of a Lie group $G$. If $H$ is a connected
Lie subgroup of $G$ with $gHg^{-1}=H$ for all $g\in\Gamma$, then $H$ is normal in $G$.
\end{lem}
\begin{proof}
Let $\frakh$ be the Lie algebra of $H$. From $gHg^{-1}=H$ we get $\Ad(g)\cdot\frakh=\frakh$ for $g\in\Gamma$.
Since the subgroup $\{\varphi\in\Aut(\frakg):\varphi(\frakh)=\frakh\}$ is Zariski-closed, it follows
$\Ad(x)\cdot\frakh=\frakh$ for all $x\in G$. And as $H$ is connected, we obtain $xHx^{-1}=H$ for all
$x\in G$.
\end{proof}

Images of $\Gamma$ in quotients of $G$ inherit algebraic density.

\begin{lem}\label{L3_2}
Let $S$ be a normal Lie subgroup of a Lie group $G$. If $\Gamma$ is an algebraically dense
subgroup of $G$, then $\Gamma S/S$ is an algebraically dense subgroup of $G/S$.
\end{lem}
\begin{proof}
We define $\bar{G}=G/S$, $\bar{\frakg}=\frakg/\fraks$ and
 $\Aut(\frakg,\fraks)=\{\varphi\in\Aut(\frakg):\varphi(\fraks)=\fraks\}$. Clearly the natural
map $\pi:\Aut(\frakg,\fraks)\to\Aut(\bar{\frakg})$ is a homomorphism of algebraic groups.
Now the assumption $\Ad(G)\subset\underline{\Ad(\Gamma)}$ implies
\[\Ad(\bar{G})=\pi(\Ad(G))\subset\pi(\,\underline{\Ad(\Gamma)}\,)=\underline{\pi(\Ad(\Gamma))}\subset\underline{\Ad(\bar{\Gamma})}\;.\]
\end{proof}

\begin{lem}\label{L3_3}
Let $G$ be an arbitrary Lie group. If $\Gamma$ is an algebraically dense solvable subgroup of $G$
contained in the image of the exponential map, then $G$ is solvable.
\end{lem}
\begin{proof}
Put $\fraks=\RR-\lspan\{Y\in\frakg:\exp(Y)\in\Gamma\}$. Now $\exp(\Ad(g)\cdot Y)=g\exp(Y)g^{-1}\in\Gamma$ for
all $Y\in\fraks$ implies $\Ad(g)\cdot\fraks=\fraks$ for all $g\in\Gamma$ and hence $\Ad(x)\cdot\fraks=\fraks$
for all $x\in G$ because $\Gamma$ is algebraically dense. Let $S$ be the connected normal subgroup of~$G$
with Lie algebra $\fraks$. The assumption $\Gamma\subset\exp(\frakg)$ implies $\Gamma\subset S$.\\

We consider the derived series of $\Gamma$ given by $\Gamma_0:=\Gamma$ and $\Gamma_{j+1}:=[\Gamma_j,\Gamma_j]$.
We have $\Gamma=\Gamma_0\supset\Gamma_1\supset\ldots\supset\Gamma_{r-1}\supset\Gamma_r=\{e\}$ with $\Gamma_j/\Gamma_{j+1}$
abelian. Now we put $S_0:=S$ and define $S_{j+1}:=[S_j,S_j]$ as the connected Lie subgroup generated by
commutators of elements of $S_j$. It holds $\Gamma_j\subset S_j$ and $\Ad(S_0)\subset\Ad(G)\subset\underline{\Ad(\Gamma)}$.
By induction we obtain
\begin{align*}
\Ad(S_{j+1})&=\Ad([S_j,S_j])=[\Ad(S_j),\Ad(S_j)]\subset[\underline{\Ad(\Gamma_j)},\underline{\Ad(\Gamma_j)}]\\
&=\underline{[\Ad(\Gamma_j),\Ad(\Gamma_j)]}=\underline{\Ad([\Gamma_j,\Gamma_j])}=\underline{\Ad(\Gamma_{j+1})}\;.
\end{align*}
This shows that $\Ad(S_r)\subset\underline{\Ad(\Gamma_r)}=\{\Id\}$ is trivial so that $S_r\subset Z(G)$ is abelian.
Since the quotient $S_j/S_{j+1}$ is abelian for $0\le j\le{r-1}$, it follows that $S$ is solvable. Moreover,
we have $\underline{\Ad(G)}=\underline{\Ad(\Gamma)}=\underline{\Ad(S)}$ which shows that $\Ad(G)$ and hence $G$
are solvable.
\end{proof}

The crucial condition $\Ad(S_j)\subset\underline{\Ad(\Gamma_j)}$ in the proof of Lemma~\ref{L3_3} means that
$\Gamma_j$ is an algebraically dense subgroup of~$S_j$.\\

We observe the following rigidity phenomenon: If $\Gamma$ is a solvable algebraically dense subgroup of $G$
with $\Gamma\subset\exp(\frakg)$ and derived series $\Gamma=\Gamma_0\supset\ldots\supset\Gamma_r=\{e\}$, then
there exists a series $S=S_0\supset\ldots\supset S_r$ of connected normal Lie subgroups such that $S_j/S_{j+1}$
is abelian and $\Gamma_j$ is algebraically dense in $S_j$.\\

\begin{lem}\label{L3_4}
Let $\Gamma$ be a uniform subgroup of a simply connected nilpotent Lie group~$G$. Then $\Gamma$
is algebraically dense in~$G$.
\end{lem}
\begin{proof}
First, $\Ad(G)$ is simply connected and Zariski-closed as a connected unipotent Lie subgroup
of $\GL(\frakg)$. Let $\overline{\Ad(\Gamma)}$ be the closure of $\Ad(\Gamma)$ with respect to the (Euclidean)
topology of $G$. Then $\overline{\Ad(\Gamma)}\setminus\Ad(G)$ is a compact Hausdorff space because
$\Gamma\setminus G$ is compact and the natural map $\Gamma\setminus G\to\overline{\Ad(\Gamma)}\setminus\Ad(G)$
is a continuous surjection. Since $\Ad(\Gamma)$ and $\overline{\Ad(\Gamma)}$ have the same
Zariski-closure, it follows $\underline{\Ad(\Gamma)}=\underline{\Ad(G)}$ from Proposition~\ref{P2_1}.
Thus $\Gamma$ is algebraically dense.
\end{proof}

Algebraic density is sufficient for the existence of syndetic hulls.

\begin{thm}\label{T3_1}
Let $G$ be a connected solvable Lie group and $\Gamma$ an algebraically dense closed subgroup
of~$G$. Then $\Gamma$ has a syndetic hull in~$G$. The syndetic hull of $\Gamma$ is unique if
$G$ is simply connected.
\end{thm}
\begin{proof}
Let $p:\tilde{G}\to G$ be the universal covering of $G$ and $\tilde{\Gamma}=p^{-1}(\Gamma)$. Identifying
the Lie algebras of $\tilde{G}$ and $G$, we see that $\Ad(\tilde{G})=\Ad(G)$ and $\Ad(\tilde{\Gamma})=\Ad(\Gamma)$
because $\ker p\subset Z(\tilde{G})$. Thus $\tilde{\Gamma}$ is algebraically dense in $\tilde{G}$. If $\tilde{S}$
is a syndetic hull of $\tilde{\Gamma}$ in $\tilde{G}$, then $S:=p(\tilde{S})$ is a syndetic hull of $\Gamma$
in $G$.\\

So let $G$ be simply connected. Let $G':=[G,G]$ and $\Gamma':=[\Gamma,\Gamma]$. Note that
$\Gamma'$ is algebraically dense in~$G'$. Since $G'$ is a simply connected nilpotent Lie
group, there is a unique syndetic hull $S'$ of $\Gamma'$ in $G'$ by Theorem~\ref{T2_1}. As
in the proof of Theorem~\ref{T2_2} we obtain $\Gamma\subset N_G(S')$ and hence $\Ad(g)\fraks'=\fraks'$
for all $g\in\Gamma$. Since $\Gamma$ is algebraically dense in $G$, it follows $\Ad(x)\fraks'=\fraks'$
for all $x\in G$. This means that $S'$ is normal in $G$ because $S'$ is connected.\\

The quotient $\bar{G}=G/S'$ is again a simply connected solvable Lie group. We claim that $\bar{G}$
contains $\bar{\Gamma}=\Gamma S'/S'$ as a closed subgroup: Since $\Gamma\cap S'$ is uniform in~$S'$, it
follows that $\Gamma\setminus\Gamma S'\cong (\Gamma\cap S')\setminus S'$ is compact and hence closed
in~$\Gamma\setminus G$. Consequently, its preimage $\Gamma S'$ under the projection $G\to\Gamma\setminus G$
is closed in $G$. This proves that $\bar{\Gamma}$ is closed in $\bar{G}$. By Lemma~\ref{L3_2} we
know that $\bar{\Gamma}$ is algebraically dense in~$\bar{G}$. For $g_1,g_2\in\Gamma$ and $x_1,x_2\in S'$
we find that
\[[g_1x_1,g_2x_2]=(g_1x_1g_1^{-1})(g_1g_2x_2x_1^{-1}g_2^{-1}g_1^{-1})(g_1g_2g_1^{-1}g_2^{-1})
(g_2x_2^{-1}g_2^{-1})\]
lies in $S'$. This shows $[\Gamma S',\Gamma S']\subset S'$ which means that $\bar{\Gamma}$ is abelian.
Since $\Ad(\bar{G})$ is contained in the Zariski closure of the abelian group $\Ad(\bar{\Gamma})$,
it follows that $\Ad(\bar{G})$ is also abelian. This means that $\bar{G}$ is a simply connected
2-step nilpotent Lie group. By Theorem~\ref{T2_1} we know that there exists a unique
syndetic hull $\bar{S}$ of $\bar{\Gamma}$ in $\bar{G}$. By Lemma~\ref{L1_1} the closed subgroup
$S:=q^{-1}(\bar{S})$ of $G$ is connected because $\bar{S}$ and $S'$ are connected. As in step 2.\
of the proof of Theorem~\ref{T2_2} we conclude that $\Gamma\setminus G$ is compact using that
$\Gamma'\setminus G'$ and $\bar{\Gamma}\setminus\bar{G}$ are compact. Thus $S$ is a syndetic hull
of $\Gamma$ in $G$.\\

It remains to prove uniqueness. Let $S_1,S_2$ be syndetic hulls of $\Gamma$ in $G$. Since
$\Gamma\subset S_1$, we have $\Gamma'=[\Gamma,\Gamma]\subset[S_1,S_1]$. Since $[S_1,S_1]$ is connected,
it follows $S'\subset[S_1,S_1]\subset S_1$. Similarly $S'\subset S_2$. Clearly $\bar{S}_1=S_1/S'$ and
$\bar{S}_2=S_2/S'$ are syndetic hulls of $\bar{\Gamma}$ in $\bar{G}$. By Theorem~\ref{T2_1} it
follows $\bar{S}_1=\bar{S}_2$ and hence $S_1=S_2$.
\end{proof}

\vspace{0.5cm}

\begin{defn}\label{D3_2}
A subgroup $\Gamma$ of a Lie group $G$ is called algebraically connected if the Zariski closure
of~$\Ad(\Gamma)$ in~$\Aut(\frakg)$ is connected in the Euclidean topology, i.e., if
$\underline{\Ad(\Gamma)}=\underline{\Ad(\Gamma)}^{\circ}$.
\end{defn}

Algebraic connectedness is sufficient for the existence of syndetic hulls.

\begin{prop}\label{P3_1}
Let $G$ be a solvable Lie group with connected center. If $\Gamma$ is an algebraically connected
subgroup of $G$, then $\Gamma$ has a syndetic hull in~$G$.
\end{prop}
\begin{proof}
Let $H$ denote the $\Ad$-preimage of $\underline{\Ad(\Gamma)}$. Since $Z(G)$ is connected, we know
that $H$ is a closed connected subgroup of $G$. By definition $\Gamma$ is an algebraically dense closed
subgroup of~$H$. By Theorem~\ref{T3_1} it follows that $\Gamma$ possesses a syndetic hull in $H$
and hence in~$G$.
\end{proof}

Recall that $Z(G)$ is connected provided that $G$ is exponential.

\end{document}